\documentclass[10pt]{amsart}
\def\cal#1{\mathcal{#1}}

          \usepackage{amssymb}
          \usepackage{pgf} 
          \usepackage{amsmath}
          \usepackage{amsthm}
          \usepackage{amsfonts}
          \usepackage{enumerate}
          \usepackage{url}
          \usepackage{color}

\usepackage[numbers,sort&compress]{natbib}
\def\cal{\mathcal}
\newcommand{\comment}[1]{}

\newcommand{\proba}{\mathbb P}
\newcommand{\esp}{{\mathbb E}}

\newcommand{\defe}{\mathrel{\mathop:}=}
\newcommand{\inv}{^{-1}}

\newcommand{\calB}{{\cal B}}

\def\B{{\mathbb B}}
\def\C{{\mathbb C}}

\def\E{{\mathbb E}}

\def\Z{{\mathbb Z}}

\newcommand{\eqnh}{\begin{eqnarray*}}
\newcommand{\eqne}{\end{eqnarray*}}
\newcommand{\eqnhn}{\begin{eqnarray}}
\newcommand{\eqnen}{\end{eqnarray}}
\newcommand{\equh}{\begin{equation}}
\newcommand{\eque}{\end{equation}}

\def\prodd#1#2#3{\prod_{#1 = #2}^{#3}}

\newcommand{\eqd}{\stackrel{\rm d}{=}}
\def\eqfdd{\stackrel{\rm f.d.d.}=}

\def\topp#1{^{(#1)}}

\def\ccbb#1{\left\{#1\right\}}

\def\sccbb#1{\{#1\}}

\def\spp#1{(#1)}
\def\pp#1{\left(#1\right)} 
 
\def\bb#1{\left[#1\right]}

\def\qmand{\quad\mbox{ and }\quad}

\def\qmwith{\quad\mbox{ with }\quad}
\def\mfa{\mbox{ for all }}

\def\wt#1{\widetilde{#1}}
\def\wb#1{\overline{#1}}

\def\weakto{\Rightarrow}

\def\Z{{\mathbb Z}}

\def\R{{\mathbb R}}

\def\N{{\mathbb N}}

\usepackage{epsfig}
\usepackage{ifthen}

\newtheorem{Thm}{Theorem}[section]

\newtheorem{Prop}[Thm]{Proposition}

\theoremstyle{definition}
\newtheorem{Rem}[Thm]{Remark}

\numberwithin{equation}{section}

\newcommand{\RR}{\mathbb{R}}
\newcommand{\CC}{\mathbb{C}}
\newcommand{\eps}{\varepsilon}

\usepackage{color,soul}
\renewcommand{\comment}[1]{{\color{blue}\fbox{#1}}}

\newcommand{\longcommenthide}[1]{}

\title{The local structure of $q$-Gaussian processes}
\author{W\l odzimierz Bryc}
\address
{
W\l odzimierz Bryc\\
Department of Mathematical Sciences\\
University of Cincinnati\\
2815 Commons Way\\
Cincinnati, OH, 45221-0025, USA.
}
\email{wlodzimierz.bryc@uc.edu}

\author{Yizao Wang}
\address
{
Yizao Wang\\
Department of Mathematical Sciences\\
University of Cincinnati\\
2815 Commons Way\\
Cincinnati, OH, 45221-0025, USA.
}
\email{yizao.wang@uc.edu}
\begin{document}\sloppy

\date{\today. File: \jobname.tex}


\begin{abstract}
The local structure  of
$q$-Ornstein--Uhlenbeck process  and $q$-Brownian motion are investigated, for all $q\in(-1,1)$. These are classical Markov processes that arose from the study of noncommutative probability. These processes have discontinuous sample paths, and the local small jumps are characterized by tangent processes.
It is shown that for all $q\in(-1,1)$, the tangent processes at
the interior of the state space
are scaled Cauchy processes possibly with drifts. The tangent processes at the boundary of the
state space are also computed, but they are not well known processes in classical probability theory.
Instead, they can be associated to the free $1/2$-stable law,
 a well known distribution in free probability, via Biane's construction.
\end{abstract}
\maketitle

\section{Introduction}

In this paper, we investigate the trajectories, particularly the jumps, of certain Markov processes that recently have drawn %
interest from both classical and noncommutative probability communities. These processes, known as  $q$-Gaussian processes,
  arose from the intriguing connection between noncommutative probability and classical probability %
  described in the seminal work by \citet{bozejko97qGaussian}. Since then, the connection has motivated many advances on both Markov processes and their counterparts in noncommutative probability, see for example  \citep{bozejko06class,bryc05conditional,bryc01stationary,anshelevich13generators}.
   Here, we take the classic probability point of view, and we are interested in  the local structure of these Markov processes.

In particular, we focus on
 the so-called $q$-Ornstein--Uhlenbeck process and $q$-Brownian motion for  $q\in(-1,1)$.
The marginal distribution of the $q$-Ornstein--Uhlenbeck process is a symmetric probability measure supported on closed interval $-2/\sqrt{1-q}\leq x\leq 2/\sqrt{1-q}$ and has
  probability density function
 \equh\label{eq:p*}%
p(x) = \frac{\sqrt{1-q}\cdot(q)_\infty}{2\pi}\sqrt{4-(1-q)x^2}\prodd k1\infty\bb{(1+q^k)^2-(1-q)x^2q^k},
\eque
where $(q)_\infty = \prodd k1\infty(1-q^k)$.
This distribution is sometimes called the $q$-normal distribution and appears also  as the orthogonality measure of the   $q$-Hermite polynomials \citep[Section 13.1]{ismail09classical}.
The marginal distribution of the $q$-Brownian motion is just a dilation of \eqref{eq:p*} due to the relation \eqref{eq:OU-Bm} below.

 It is known that the $q$-normal distribution
 interpolates
between several important distributions as $q$ varies between $-1$ and $1$.
 When $q=0$ it becomes the celebrated
 Wigner semicircle law that plays a fundamental role in random matrix theory,
when  $q$ goes to $1$ it converges to the
 standard normal distribution that is ubiquitous in classical probability theory, and
when $q$ goes to $-1$ it converges to the symmetric discrete distribution on $\{\pm1\}$ which is sometimes called the Rademacher distribution.

Next, at the process level the transition probabilities of the two Markov processes were identified in  \citep[Theorem 4.6]{bozejko97qGaussian}.
\citet[Section 4]{szablowski12qWiener}
pointed out that they are Feller  Markov processes with a %
{\em c\`adl\`ag}
(right-continuous and left-limit exists) version.
Throughout the paper we consider only {\em c\`adl\`ag} versions %
of stochastic processes.
 One can also easily show that
as $q$ goes to $1$ the two processes converge
in distribution
to Brownian motions and Ornstein--Uhlenbeck processes respectively.
This roughly says that the paths of $q$-Brownian motion
are close to those of
the Brownian motion for $q$ close to 1, as illustrated by Figure~\ref{fig:qBm}.
However, while it is well known that Brownian motions have almost surely continuous paths (e.g.~\citep{morters10brownian}),
 it has been a folklore that the trajectories of $q$-Brownian motions have jumps, as can also be seen from Figure~\ref{fig:qBm}. Our motivation is to
understand better these jumps, and hence also the trajectories of $q$-Gaussian processes.
\begin{figure}
\includegraphics[width=.7 \textwidth]{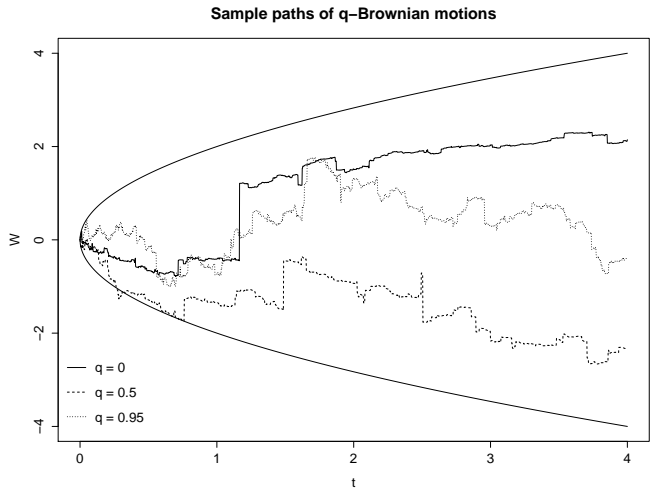}
\label{fig:qBm}  \caption{Three trajectories of discretized (in time) $q$-Brownian motions $\{W_{4i/n}^{(q)}\}_{i=0,\dots,n}$  with $q = 0, 0.5$ and $0.95$ respectively, with $n = 2000$, simulated by R. The solid parabolic line ($w^2=4t$) is the boundary of the
support
of the free Brownian motion ($q=0$).}
\end{figure}

In this paper, we use the notion of tangent process \cite{falconer03local} to characterize the local structure of $q$-Gaussian
processes, and confirm that while large jumps become unlikely for $q$ close to 1, see Remark \ref{rem:big-jumps},
 the two processes are locally approximated by    the Cauchy process  for every fixed $q<1$,
 with a possible drift and a multiplicative constant depending on $q$.
In order to accomplish our goal, we  modify slightly
a general framework  from  \citet{falconer03local} to allow for dependence on  location  at time $s\geq 0$. Namely,
   let
   $Z = \{Z_t\}_{t\ge0}$ be a {\em c\`adl\`ag} Markov process, for $s>0$ and $x$ in the support of
   random variable
   $Z_s$, let $\proba(\cdot\mid Z_s = x)$ be
     the law of the Markov process conditioning on $Z_s = x$,
  and we say that $\zeta=\{\zeta_t\}_{t\ge0}$ is a {\em tangent process of $Z$ at time $s$ and location $x$},
if under the law  $\proba(\cdot\mid Z_s = x)$ we have
weak convergence
\equh\label{eq:tangent}
\ccbb{\frac{Z_{s+\epsilon t}-Z_s}{\epsilon^\beta}}_{t\ge 0}\weakto
\{\zeta_t\}_{t\ge0}
\eque
as $\epsilon\downarrow 0$, for some $\beta>0$ appropriately chosen,  in $D([0,\infty))$ equipped with Skorohod topology.
While $Z_{s+\epsilon}$ converges to $Z_s$ in probability as $\epsilon\downarrow 0$ for a {\em c\`adl\`ag} process $Z$, the tangent process in~\eqref{eq:tangent} provides information on the rates and local fluctuations of the convergence.
To establish~\eqref{eq:tangent}, for the two processes it suffices to work with the conditional transition probability densities of $\{Z_t\}_{t\ge s}$ given $Z_s$, so that the left-hand side induces a unique probability measure on $D([0,\infty))$ \citep{szablowski12qWiener}.
When the tightness is difficult to establish, we consider only convergence of finite-dimensional distributions.

Our main results consist of
identifying tangent processes for both $q$-Brownian motions and $q$-Ornstein--Uhlenbeck processes, denoted by $W\topp q = \{W_t\topp q\}_{t\ge0}$  and $X\topp q = \{X_t\topp q\}_{t\in\R}$ respectively throughout the paper.
It is well known that for the same $q\in(-1,1)$, these two processes can be mapped onto each other by a deterministic  transformation
\equh\label{eq:OU-Bm}
\ccbb{X_t\topp q}_{t\in\R}= \ccbb{e^{-t}W\topp q_{e^{2t}}}_{t\in\R}.
\eque
It is more convenient to work with the $q$-Ornstein--Uhlenbeck process as it is a
stationary Markov process
on the state space
 $[-2/\sqrt{1-q},2/\sqrt{1-q}]$.
Our
 findings are summarized as follows.\medskip

 {(i)} For
 $q$-Ornstein--Uhlenbeck process, we first prove  that for all $q\in(-1,1)$,  the tangent process in~\eqref{eq:tangent} exists at all location $x\in(-2/\sqrt{1-q},2/\sqrt{1-q})$ with $\beta = 1$, and is a Cauchy process up to a multiplicative constant (Theorem~\ref{thm3.1}). In other words, locally the $q$-Ornstein--Uhlenbeck process
  behaves like a Cauchy process, for all $q\in(-1,1)$. It is somehow surprising to see that, although the
 local jumps disappear in the limit
 as $q\to 1$,
 they persist in such a qualitative manner.\medskip

{(ii)}
We investigate the tangent process of $q$-Ornstein--Uhlenbeck process at the left boundary point of the
state space
 $x_- = -2/\sqrt{1-q}$.
 In this case, the tangent process still exists as in~\eqref{eq:tangent}, but with scaling parameter $\beta =2$, and is a different Markov process (Proposition~\ref{prop:qOU2}).\medskip

{(iii)} The Markov process obtained as the tangent process at the boundary point seems to have not been well investigated in classical probability theory, to the best of our knowledge. Instead,
somehow unexpectedly,  we identify this process as the Markov process (up to a quadratic drift) associated to the free $1/2$-stable law via the construction of \citet{biane98processes},
after whom we name the process {\it $1/2$-stable Biane process} (Proposition~\ref{P:biane}). This connection is irrelevant to the path properties of the processes, but
it
is of its own interest.\medskip

{(iv)} For the $q$-Brownian motion, since it is not stationary and has inhomogeneous transition probabilities, the situation is slightly
more subtle.
The tangent process of the $q$-Brownian motion at
 the interior of the support of $W\topp q_s$
 is still Cauchy, but with a linear drift (Proposition~\ref{prop:qBm1}). The tangent process at time $s$ at the boundary of the
 support $x_- = -2\sqrt{s/(1-q)}$ this time however, instead of in the common form~\eqref{eq:tangent}, appears as the limit of
\[
\ccbb{\frac{W\topp q_{s+t\epsilon}+(s(1-q))^{-1/2}\cdot t\epsilon-W\topp q_s}{\epsilon^2}}_{t\ge0}
\]
as $\epsilon\downarrow 0$ under the law $\proba(\cdot\mid W_s\topp q = x_-)$ (Proposition~\ref{prop:qBm2}). The tangent process turns out to be
the $1/2$-stable Biane process up to a multiplicative constant.
\medskip

The paper is organized as follows.
Section~\ref{sec:tangent} establishes
limit theorems for the tangent processes at both inner and boundary points for both processes.
The connection
to noncommutative probability, and particularly the identification of
the $1/2$-stable Biane process, are provided in Section~\ref{sec:biane} in a self-contained manner. %

\section{Convergence to tangent processes}\label{sec:tangent}
We first introduce the two processes that appear, with appropriate scalings and drifts, as the tangent processes of $q$-Gaussian processes.
Both processes are Markov processes. The first is Cauchy process (symmetric 1-stable L\'evy process), starting from $0$  with transition probability density
\[
f\topp{1}_{t_1,t_2}(y_1,y_2) =
\frac{1}{\pi}\frac{t_2-t_1}{(y_2-y_1)^2+(t_2-t_1)^2}, \;  0\leq t_1<t_2<\infty,y_1,y_2\in\R.
\]
The second also starts from 0, and has transition probability  density
\begin{multline}\label{eq:biane}
f^{(1/2)}_{t_1,t_2}(y_1,y_2)%
= \frac{(t_2-t_1)\sqrt{4y_2-t_2^2}}{2\pi\bb{(y_2-y_1)^2-(t_2-t_1)(t_1y_2-t_2y_1)}},\\ 0\leq t_1<t_2<\infty, y_1>\frac{t_1^2}4, y_2>\frac{t_2^2}4.
\end{multline}
Note that the
support
of the second process at time $t$ is $[t^2/4,\infty)$.
The two processes are denoted by $\Z\topp\alpha = \{\Z\topp\alpha_t\}_{t\ge0}$ with $\alpha = 1,1/2$ respectively, and the marginal distributions are given at~\eqref{eq:Cauchy} and~\eqref{half-stable} below. Both processes are self-similar with parameter $1/\alpha$ in the sense that
\equh\label{eq:SS}
\ccbb{\Z\topp\alpha_{\lambda t}}_{t\ge0} \eqfdd \lambda^{1/\alpha}\ccbb{\Z\topp\alpha_t}_{t\ge0}, \lambda>0, \alpha = 1,1/2.
\eque
Furthermore, $\Z\topp1$ has independent and stationary increments as a L\'evy process. The process $\Z\topp{1/2}$ has non-stationary increments, but with a drift and time scaling $\{\Z\topp{1/2}_{2t}-t^2\}_{t\ge0}$ is self-similar with time-homogeneous transition probability density
\equh\label{eq:p1/2}
p\topp{1/2}_{t_1,t_2}(y_1,y_2) =
\frac{2 \left(t_2-t_1\right)  \sqrt{y_2}}{\pi
  \bb{(y_2-y_1)^2+2 (y_1+y_2)(t_2-t_1)^2+(t_2-t_1)^4}}, y_1,y_2,t_1,t_2>0.
\eque

Both processes also arise from free probability.
In particular, $\Z\topp 1$ and $\Z\topp{1/2}$ are the Markov processes associated to free $1$-stable and $1/2$-stable semigroups respectively. For the sake of simplicity,
we call $\Z\topp{1/2}$ the {\it $1/2$-stable Biane process} in the sequel. We explain this connection to free probability in the Section~\ref{sec:biane}. The discussion there is independent from the rest of %
this section,
but is of its own interest.

Below, we first %
consider the tangent processes first %
 $q$-Ornstein--Uhlenbeck processes and then of $q$-Brownian motions.
\subsection{Tangent processes of $q$-Ornstein--Uhlenbeck processes}\label{sec:qOU}
Fix $q\in(-1,1)$, and  let $X\topp q= \{X\topp q_t\}_{t\in\R}$ denote
a $q$-Ornstein--Uhlenbeck process.
That is, $X\topp q$ is a stationary Markov process with {\em c\`adl\`ag} trajectories, with
the marginal probability density function $p(x)$  given by \eqref{eq:p*}, and with the  transition probability density function $p_{s,t}(x,y)$
given by, for $x,y\in[-2/\sqrt{1-q},2/\sqrt{1-q}]$,
\equh\label{eq:pst}
p_{s,t}(x,y) = (e^{-2(t-s)};q)_\infty \prodd k0\infty\frac1{\varphi_{q,k}(t-s,x,y)} \cdot p(y),
\eque
with
\[
\varphi_{q,k}(\delta,x,y) = (1-e^{-2\delta}q^{2k})^2 - (1-q)e^{-\delta}q^k(1+e^{-2\delta}q^{2k})xy + (1-q)e^{-2\delta}q^{2k}(x^2+y^2).
\]
Here and below, we write
\[
(a;q)_\infty := \prodd k0\infty(1-aq^k),\mfa  a\in\R, q\in(-1,1).
\]
The above densities can be found at \citep[Corollary 2]{bryc05probabilistic} and \citep[Eq.(2.9)]{szablowski12qWiener}.

Two bounds on $\varphi_{q,k}$ are useful. First, observe that $\varphi_{q,k}(\delta,x,y) = a(x^2+y^2)-bx + c$ has a quadratic form with $b/2a>1$. Thus,
\[
\min_{|x|,|y|\leq \frac2{\sqrt{1-q}}}\varphi_{q,k}(\delta,x,y) = (1-e^{-\delta}q^k)^4\geq (1-|q|^k)^4, k\in\N_0, \delta>0.
\]
At the same time,
\begin{align}
\varphi_{q,0}(\delta,x,y) & = (1-e^{-2\delta})^2 - (1-q)e^{-\delta}(1+e^{-2\delta})xy + (1-q)e^{-2\delta}(x^2+y^2)\nonumber\\
& = e^{-2\delta}\bb{4\sinh^2(\delta) + (1-q)(x-y)^2 + 2(1-q)xy(1-\cosh(\delta))}\label{eq:varq0}\\
& \geq e^{-2\delta}\bb{4\sinh^2(\delta) + 8(1-\cosh(\delta)) + (1-q)(x-y)^2} \nonumber\\
& =  e^{-2\delta}\bb{16\sinh^4(\delta/2) + (1-q)(x-y)^2},\nonumber
\end{align}
where in the inequality above we used the fact that $1-\cosh(\delta)\leq 0$. In particular,
\equh\label{eq:minphi}
\prodd k0\infty\frac1{\varphi_{q,k}(\delta,x,y)}\leq \frac{e^{2\delta}}{[16\sinh^4(\delta/2)+(1-q)(x-y)^2] (|q|)_\infty^4}.
\eque

We first look at the tangent process at the
interior of the state space.
Consider the process
\[
Y\topp\epsilon_t \defe \frac{X\topp q_{\epsilon t}-X\topp q_0}{\epsilon}.
\]

\begin{Thm}\label{thm3.1}
For all %
 $q\in(-1,1)$, $x\in(-2/\sqrt{1-q},2/\sqrt{1-q})$,
under
$\proba(\cdot\mid X_0\topp q=x)$, we have
\[
\ccbb{Y\topp\epsilon_t}_{t\in[0,\infty)}\weakto c_{q,x}\ccbb{\Z\topp 1_t}_{t\in[0,\infty)} \qmwith c_{q,x} = \sqrt{\frac4{1-q}-x^2}
\]
in $D([0,\infty))$ as $\epsilon\downarrow 0$,
where $\Z\topp1$ is the Cauchy process.
\end{Thm}
\begin{proof}
For $x\in(-2/\sqrt{1-q},2/\sqrt{1-q})$, let
 $\wt p_{t_1,t_2}\topp{\epsilon,x}(y_1,y_2)$ denote the transition probability density function of $Y\topp\epsilon$, conditioning on $X_0\topp q = x$. Then, writing $\delta:=t_2-t_1$,
\begin{multline}\label{eq:pdfY}
\wt p\topp{\epsilon,x}_{t_1,t_2}(y_1,y_2) = p_{\epsilon t_1,\epsilon t_2}(x+y_1\epsilon, x+y_2\epsilon)\epsilon\\
= \frac{\epsilon\sqrt{1-q}\cdot(e^{-2\epsilon\delta};q)_\infty(q)_\infty}{2\pi}
\frac{\sqrt{4-(1-q)(x+y_2\epsilon)^2}}{\varphi_{q,0}(\epsilon\delta,x+y_1\epsilon,x+y_2\epsilon)}\prodd k1\infty \frac{\psi_{q,k}(x+y_1\epsilon)}{\varphi_{q,k}(\epsilon\delta,x+y_1\epsilon,x+y_2\epsilon)}
\end{multline}
with $\psi_{q,k} (x)= (1+q^k)^2-(1-q)x^2q^k$. We factorize $\wt p_{t_1,t_2}\topp{\epsilon,x}$ in this way because in the
analysis below when computing the limiting probability densities, the infinite product is easy to deal with and contributes
asymptotically only a constant, while the square-root term and $\varphi_{q,0}$ contribute to the limiting density and are
treated separately. This pattern of calculations will repeatedly show up in all derivations of tangent processes below.
\medskip

 (i) We first prove the convergence of finite-dimensional distributions.   For this purpose, by Scheff\'e's theorem (\cite[Theorem 16.12]{billingsley95probability}) it suffices to prove pointwise convergence of joint probability densities.
In particular, we  prove
\equh\label{eq:ptwise}
\lim_{\epsilon\downarrow0}\wt p_{t_1,t_2}\topp{\epsilon,x}(y_1,y_2) = f\topp1_{t_1,t_2}\pp{\frac{y_1}{c_{q,x}},\frac{y_2}{c_{q,x}}}\frac1{c_{q,x}}.
\eque
  Write $\delta:=t_2-t_1$. Observe that as $\epsilon\downarrow0$,
recalling~\eqref{eq:varq0},
\[
\varphi_{q,0}(\epsilon\delta,x+y_1\epsilon,x+y_2\epsilon) %
\sim \epsilon^2\bb{(1-q)(y_2-y_1)^2 + \delta^2(4-(1-q)x^2)},
\]
and
\begin{align*}
\prodd k1\infty & \frac{\psi_{q,k}(x+y_1\epsilon)}{\varphi_{q,k}(\epsilon\delta, x+y_1\epsilon, x+y_2\epsilon)} \\
& \sim \prodd k1\infty \frac{(1+q^k)^2 - (1-q)x^2q^k}{(1-q^{2k})^2-(1-q)q^k(1+q^{2k})x^2+2(1-q)q^{2k}x^2}\\
& = \prodd k1\infty \frac{(1+q^{k})^2 - (1-q)x^2q^k}{(1-q^k)^2(1+q^k)^2 - (1-q)x^2q^k(1-q^k)^2} = \prodd k1\infty\frac1{(1-q^k)^2}.
\end{align*}
Here $y_1,y_2\in\R$ are fixed.  To pass to the limit, we use the fact that $|q|<1$ and $|x|<2/\sqrt{1-q}$ so the product is bounded  by a convergent product uniformly  over all small enough $\epsilon$. Small enough means  that  $\sqrt{1-q}|x+y_1\epsilon|<2$ and
 $\sqrt{1-q}|x+y_2\epsilon|<2$.

Finally, we note that
$$(e^{-2\epsilon\delta};q)_\infty=(1-e^{-2\epsilon\delta})\prod_{k=1}^\infty(1-e^{-2\epsilon\delta}q^k) \sim 2
 \epsilon \delta \prod_{k=1}^\infty(1- q^k)= 2
 \epsilon \delta (q)_\infty.$$
Combining all the calculation above, we have thus shown
\[
\lim_{\epsilon\downarrow0} p_{\epsilon t_1,\epsilon t_2}(x+y_1\epsilon,x+y_2\epsilon)\epsilon
= \frac{\delta\sqrt{1-q}\sqrt{4-(1-q)x^2}}{\pi[(1-q)(y_2-y_1)^2+\delta^2(4-(1-q)x^2)]},
\]
which is the same as~\eqref{eq:ptwise}.

Since our Markov processes start at $0$, the univariate densities also converge,  as they  are just the transition densities evaluated at $y_1=0$ and $t_1=0$.
We  have thus proved the convergence of finite-dimensional distributions.
\medskip

\noindent(ii) Next we prove the tightness.
We show that, for all $0\leq t_1<t_2\leq T<\infty$, $\epsilon>0$, independent of $x$ and $y_1$,
\equh\label{eq:aldous}
\esp_x(|Y_{t_2}\topp\epsilon-Y_{t_1}\topp\epsilon|^2\wedge 1\mid Y_{t_1}\topp\epsilon = y_1) \leq C_{T,q}(t_2-t_1)
\eque
for some constant $C_{T,q}$ depending only on $T$ and $q$. Then, the tightness of the processes $\{Y\topp\epsilon\}_{\epsilon>0}$ under the measure $\proba_x$ follows from \citet[Chapter 3, Theorem 8.6, Remark 8.7]{ethier86markov}. In particular, conditions (8.29) and (8.33) therein are satisfied for our processes.

To prove~\eqref{eq:aldous},
recall~\eqref{eq:minphi}.
It then follows that there exists a constant $C_{T,q}$ such that for all $0\leq t_1<t_2\leq T, \epsilon>0$, uniformly in $x,y_1,\epsilon$,
\[
\wt p^{x,\epsilon}_{t_1,t_2}(y_1,y_2)\leq C_{T,q}\frac{t_2-t_1}{[(y_2-y_1)^2+16\sinh^4(\epsilon(t_2-t_1)/2)/(\epsilon^2(1-q))]},
\]
whence
\begin{multline*}
\esp_x(|Y_{t_2}\topp\epsilon-Y_{t_1}\topp\epsilon|^2\wedge 1\mid Y_{t_1}\topp\epsilon = y_1)\\
\leq C_{T,q}(t_2-t_1) + \int_{|z|>1}C_{T,q}(t_2-t_1)\frac1{z^2}dz = C_{T,q}(t_2-t_1).
\end{multline*}
We have thus proved~\eqref{eq:aldous} and the tightness.
\end{proof}

The proof of Theorem \ref{thm3.1} does not apply to the boundary %
points
$x=\pm 2/\sqrt{1-q}$. At the same time, as $x\to \pm 2/\sqrt{1-q}$,
we have
$c_{q,x}\to 0$. These observations  raise the question on the tangent process at the boundary, and suggest that for a non-degenerate limit to exist we need to work with a different scaling.
Consider the process
\[
\widetilde Y\topp\epsilon_t \defe \frac{X_{\epsilon t}-X_0}{\epsilon^2}.
\]
  Let $x_- = -2/\sqrt{1-q}$ denote the left boundary point.
 \begin{Prop}\label{prop:qOU2}
   \label{prop:3.2} For all $q\in(-1,1)$, under $\proba(\cdot\mid W_0\topp q = x_-)$,
   \[
  \ccbb{\wt Y\topp\epsilon_t}_{t\ge0}\stackrel{\rm f.d.d.}
  \Longrightarrow\frac4{\sqrt{1-q}}\ccbb{\Z_{t}\topp{1/2}-\frac{t^2}4}_{t\ge0}
   \]
as $\epsilon\downarrow 0$,   where $\Z\topp{1/2}$ is the $1/2$-stable Biane process %
with transition probability densities~\eqref{eq:biane}.

 \end{Prop}
 \begin{Rem}
 Here and in Propositions~\ref{prop:qBm1} and~\ref{prop:qBm2}, we only
 prove the convergence of finite-dimensional distributions.
 \end{Rem}
  \begin{proof}
    The proof is similar to the first part of the proof of Theorem \ref{thm3.1} and consists of verification that the transition density converges. The transition density of $\{ \widetilde Y\topp\epsilon_t\}_{t\geq 0}$ is
    $p_{\epsilon t_1,\epsilon t_2}(x_-+y_1\epsilon^2, x_-+y_2\epsilon^2)\epsilon^2$. With $\delta=t_2-t_1$, the density factors   as in
    \eqref{eq:pdfY} with $x$ replaced by $x_-$, and we compute the corresponding terms one by one. The infinite product
$$     (q)_\infty\prodd k1\infty \frac{(1-q^ke^{-2\epsilon\delta}) \psi_{q,k}(x_-+y_1\epsilon^2)}{\varphi_{q,k}(\epsilon\delta,x_-+y_1\epsilon^2,x_-+y_2\epsilon^2)}
$$
converges again to $1$ as $\epsilon\downarrow0$, and the factor
\[
    \frac{\epsilon^2\sqrt{1-q}\cdot(1-e^{-2\epsilon\delta})}{2\pi}
\frac{\sqrt{4-(1-q)(x_-+y_2\epsilon^2)^2}}{\varphi_{q,0}(\epsilon\delta,x_-+y_1\epsilon^2,x_-+y_2\epsilon^2)}
 \]
contributes to the limit. As previously, $(1-e^{-2\epsilon\delta})\sim 2\epsilon\delta$, but at the boundary we have
 \[
 \sqrt{4-(1-q)(x_-+y_2\epsilon^2)^2}\sim 2\epsilon \sqrt{1-q}\sqrt{y_2},
 \]
 and
 \begin{multline*}
 \varphi_{q,0}(\epsilon\delta,x_-+y_1\epsilon^2,x_-+y_2\epsilon^2)\\
 \sim \epsilon^4 \bb{(1-q)(y_2-y_1)^2+2\sqrt{1-q}
   (y_1+y_2)(t_1-t_2)^2+(t_1-t_2)^4
 }.
 \end{multline*}
 It then follows that
 \begin{align*}
\lim_{\epsilon\downarrow0}p_{\epsilon t_1,\epsilon t_2}&(x_-+y_1\epsilon^2,x_-+y_2\epsilon^2)\epsilon^2 \\
&=
\frac{2 \left(t_2-t_1\right)  \sqrt{(1-q)y_2}}{\pi
  \bb{(1-q)(y_2-y_1)^2+2 \sqrt{1-q}   (y_1+y_2)(t_2-t_1)^2+(t_2-t_1)^4
 }}\\
& = f_{2t_1,2t_2}\topp{1/2}(\sqrt{1-q}\cdot y_1+t_1^2,\sqrt{1-q}\cdot y_2+t_2^2)\sqrt{1-q}, y_1,y_2>0.
 \end{align*}

 The desired result now follows from self-similarity~\eqref{eq:SS}.
  \end{proof}
  \begin{Rem}\label{rem:falconer}
Let $X = \{X_t\}_{t\ge0}$ be a general process. \citet{falconer03local} actually considers the annealed tangent process, $\{(X_{s+t\epsilon}-X_s)/\epsilon^\beta\}_{t\ge0}$ for $s\ge 0$, while we consider the quenched tangent process conditioning on the value of $X_s$. From our results, the annealed tangent process (without conditioning) can then be derived easily as a mixture of Cauchy process. We omit the details. The same applies to the tangent process of the $q$-Brownian motion
in Proposition~\ref{prop:qBm1}.

According to \citet{falconer03local}, for almost all time points $s$ at which the general process $X$ has a unique annealed tangent process, the tangent process must be self-similar with stationary increments. Here we have an example indicating that one cannot drop the `almost all' part of the statement. Indeed, fixing $\tau>0$ and considering $X_t = X\topp q_{\tau+t}$ with law $\proba(\cdot\mid X\topp q_\tau = x_-)$, we just showed that for this process at $s = 0$, the tangent process exists, is self-similar,  but has non-stationary increments. There is no contradiction since as discussed above, for any $s>0$ the annealed tangent process is a mixture of $\Z\topp 1$, and is thus self-similar with stationary increments.
\end{Rem}
\subsection{Tangent processes of $q$-Brownian motions}\label{sec:qBm}
In this section, consider the $q$-Brownian motion $\{W\topp q_t\}_{t\ge0}$ with transition probability density \citep[Eq.(55)]{bryc05conditional}
\equh
\label{eq:kappa}
\kappa_{t_1,t_2}\topp q(y_1,y_2)  %
 = \frac{(1-q)^{3/2}(t_2-t_1)}{2\pi}\frac{\sqrt{4t_2-(1-q)y_2^2}}{\varphi_{q,0}^*(t_1,t_2,y_1,y_2)}\prodd k1\infty\frac{\psi_{q,k}^*(t_1,t_2,y_2)}{\varphi_{q,k}^*(t_1,t_2,y_1,y_2)}
\eque
where
\[
\psi_{q,k}^*(t_1,t_2,y_2) = (t_2-t_1q^k)(1-q^{k+1})\bb{t_2(1+q^k)^2-(1-q)y_2^2q^k}, k\ge1
\]
and
\[
\varphi_{q,k}^*(t_1,t_2,y_1,y_2) = (t_2-t_1q^{2k})^2-(1-q)q^k(t_2+t_1q^{2k})y_1y_2+(1-q)(t_1y_2^2+t_2y_1^2)q^{2k}, k\ge0.
\]
We first consider the tangent process %
at the interior point of the support of $W\topp q_s$.
For $s>0$, $\epsilon>0$ consider the process
\[
\left\{V_t\topp{\epsilon,s}\right\}_{t\geq 0} := \left\{\frac{W\topp q_{s+t\epsilon} - W\topp q_s}{\epsilon}\right\}_{ t\ge 0}.
\]
\begin{Prop}\label{prop:qBm1}
For $s>0$ and $x\in(-2\sqrt{s/(1-q)}, 2\sqrt{s/(1-q)})$, under the law $\proba(\cdot \mid W\topp q_s = x)$,
\[
\ccbb{V_t\topp{\epsilon,s}}_{t\ge0}\stackrel{\rm f.d.d.}
\Longrightarrow
c_{q,s,x}\ccbb{ \Z\topp 1_t +   \frac x{\sqrt{\frac{4s}{1-q}-x^2}}\cdot t}_{t\ge0}
\]
as $\epsilon\downarrow0$,
where $\Z\topp 1$ is the  Cauchy process and
\[
c_{q,s,x} = \frac1{2s}\sqrt{\frac{4s}{1-q}-x^2}.
\]
\end{Prop}
\begin{proof}
The transition probability density of $V\topp{\epsilon,s}$ conditioning on $W_s\topp q = x$ is
\[
\kappa_{t_1,t_2}\topp{\epsilon,s,x}(y_1,y_2) := \kappa_{s+t_1\epsilon,s+t_2\epsilon}\topp q(x+y_1\epsilon,x+y_2\epsilon)\epsilon.
\]
 One can show for all fixed $x\in(-2\sqrt{s/(1-q)},2\sqrt{s/(1-q)})$,
\[
\lim_{\epsilon\downarrow0}\sqrt{4t_2-(1-q)y_2^2}=\sqrt{4s-(1-q)x^2},
\]
\begin{multline*}
\lim_{\epsilon\downarrow0}\frac{\varphi_{q,0}^*(s+t_1\epsilon,s+t_2\epsilon,x+y_1\epsilon,x+y_2\epsilon)}{\epsilon^2} \\
= (t_2-t_1)^2+(1-q)\bb{s(y_2-y_1)^2 - (t_2-t_1)(y_2-y_1)x},
\end{multline*}
and
\[
\lim_{\epsilon\downarrow0}\frac{\psi_{q,k}^*(s+t_1\epsilon,s+t_2\epsilon,x+y_2\epsilon)}{\varphi_{q,k}^*(s+t_1\epsilon,s+t_2\epsilon,x+y_1\epsilon,x+y_2\epsilon)} = \frac{1-q^{k+1}}{1-q^k}, k\ge 1.
\]
As previously, the infinite product converges uniformly in $\epsilon$ for all $\epsilon$ close enough to $0$.
It then follows that %
\begin{multline*}
\lim_{\epsilon\downarrow0}\kappa_{t_1,t_2}\topp{\epsilon,s,x}(y_1,y_2)
= \frac{1}{2\pi}\frac{\sqrt{1-q}(t_2-t_1)\sqrt{4s-(1-q)x^2}}{(t_2-t_1)^2+(1-q)[s(y_2-y_1)^2-(t_2-t_1)(y_2-y_1)x]}\\
=f\topp 1_{\tau_1,\tau_2}\pp{y_1-\frac{t_1x}{2s},y_2-\frac{t_2x}{2s}}
\end{multline*}
 with $\tau_j=\tau_j(q,s,x)=c_{q,s,x}t_j$
, $j=1,2$.
So the limiting process equals in
distribution
\[
\ccbb{\Z\topp 1_{c_{q,s,x}t}+\frac x{2s}t}_{t\ge0} \eqfdd c_{q,s,x}\ccbb{ \Z\topp 1_t +   \frac x{\sqrt{\frac{4s}{1-q}-x^2}}\cdot t}_{t\ge0}
\]
by self-similarity~\eqref{eq:SS}.
\end{proof}

Next we consider the tangent process at the boundary %
of the support.
Consider the left end-point $x_- = -2\sqrt{s/(1-q)}$ of the %
support
of the $q$-Brownian motion at time $s$, and the process
\equh\label{eq:atepsilon}
\wt V\topp{\epsilon,s}_t := \frac{W\topp q_{s+t\epsilon}+at\epsilon-W\topp q_s}{\epsilon^2} \qmwith a = -\frac2{(1-q)x_-} = \frac1{\sqrt{s(1-q)}}.
\eque
\begin{Prop}\label{prop:qBm2}
For all $s>0$, under the law $\proba(\cdot\mid W\topp q_s = x_-)$,
\[
\ccbb{\wt V\topp{\epsilon,s}_t}_{t\ge0}\stackrel{\rm f.d.d.}
\Longrightarrow
\frac1{s^{3/2}\sqrt{1-q}}\ccbb{\Z_t\topp{1/2}}_{t\ge0}
\]
as $\epsilon\downarrow 0$, where $\Z\topp{1/2}$ is the
 $1/2$-stable Biane process  %
with transition probability densities~\eqref{eq:biane}.
\end{Prop}
\begin{proof}
The transition probability density of $\wt V_t(\epsilon,s)$ under the law $\proba(\cdot\mid W\topp q_s = x_-)$ is
\[
\wb\kappa\topp{\epsilon,s}_{t_1,t_2}(y_1,y_2) = \kappa\topp q_{s+t_1\epsilon,s+t_2\epsilon}(x_--at_1\epsilon+y_1\epsilon^2,x_--at_2\epsilon+y_2\epsilon^2)\epsilon^2.
\]

Again from~\eqref{eq:kappa}, by straight-forward calculation one obtains as $\epsilon\downarrow0$,
\[
\sqrt{4(s+t_2\epsilon)-(1-q)(x_--at_2\epsilon+y_2\epsilon)^2} \sim \epsilon \sqrt{4\sqrt{s(1-q)}y_2-\frac{t_2^2}s},
\]
\begin{multline*}
\varphi_{q,0}^*(s+t_1\epsilon,s+t_2\epsilon,x_--at_1\epsilon+y_1\epsilon^2,x_--at_2\epsilon+y_2\epsilon^2)\\
\sim \epsilon^4\bb{s(1-q)(y_2-y_1)^2-\sqrt{\frac{1-q}s}(t_2-t_1)(t_1y_2-t_2y_1)},
\end{multline*}
and
\[
\lim_{\epsilon\downarrow0}\frac{\psi_{q,k}^*(s+t_1\epsilon,s+t_2\epsilon,x-a\epsilon+y_2\epsilon^2)}{\varphi_{q,k}^*(s+t_1\epsilon,s+t_2\epsilon,x-a\epsilon+y_1\epsilon^2,x-a\epsilon+y_2\epsilon^2)} = \frac{1-q^{k+1}}{1-q^k}, k\ge 1.
\]
Again, the infinite product of $\psi^*_{q,k}/\varphi^*_{q,k}$ converges uniformly for $\epsilon$ small enough as before. We thus arrive at
\begin{multline*}
\lim_{\epsilon\downarrow0}\wb\kappa\topp{\epsilon,s}_{t_1,t_2}(y_1,y_2) = \frac{\sqrt{1-q}}{2\pi}\frac{(t_2-t_1)\sqrt{4\sqrt{s(1-q)}y_2-{t_2^2}/s}}{s(1-q)(y_2-y_1)^2-\frac{\sqrt{1-q}}{\sqrt s}(t_2-t_1)(t_1y_2-t_2y_1)}\\
= f\topp{1/2}_{{t_1},{t_2}}\pp{\sqrt{s^3(1-q)}y_1,\sqrt{s^3(1-q)}y_2}\sqrt{s^3(1-q)}.
\end{multline*}
The desired result now follows.
\end{proof}
\begin{Rem}
  \label{rem:big-jumps} The tangent processes are established for fixed $q$, and they do not capture the behavior of large jumps
  as $q$ varies. To see what happens as $q$ approaches $1$,
we
 only mention here an  explicit estimate
\begin{equation}
  \label{EQ:large-jumps}
  \proba\pp{\sup_{S\leq t\leq T}|W\topp q_{t-}-W\topp q_{t}|>a}\leq \frac{(1-q)}{a^4}(T^2-S^2), \quad 0\leq S<T, a>0,
\end{equation}
which
indicates that large jumps become unlikely when $q$ is close to $1$ or when the time interval $T-S$ is small.
However, the inequality only provides an upper bound.
A precise estimate of  the asymptotic probability of large jumps will be
established in the form of a Poisson
limit theorem  in another paper.

To prove \eqref{EQ:large-jumps} we use the formula
  $$\esp(W\topp q_t-W\topp q_s)^4=(2+q)(t-s)^2+2(1-q)s(t-s), 
\quad  0\leq s<t.$$
   which can be read out from   \cite[formula (4.14)]{szablowski12qWiener}.
With   $t_i=S+i(T-S)/n$,   we have
  \begin{multline}\label{EQ:lim}
 \proba 
 \pp{\max_{i=1,\dots,n}|W\topp q_{t_i}-W\topp q_{t_{i-1}}|>a}
 \leq \frac{1}{a^4}\sum_{i=1}^n\E(W\topp q_{t_i}-W\topp
 q_{t_{i-1}})^4\\
 =\frac{2+q}{a^4}\sum_{i=1}^n (t_i-t_{i-1})^2+\frac{2(1-q)}{a^4}\sum_{i=1}^n t_{i-1}(t_i-t_{i-1})\to \frac{2(1-q)}{a^4}\int_S^T
 t\, dt \mbox{ as $n\to\infty$}.
  \end{multline}
For every trajectory,  $\max_{i=1,\dots,n}|W\topp q_{t_i}-W\topp q_{t_{i-1}}|$ converges to  $\sup_{S<t\leq T}|W\topp q_{t-}-W\topp q_{t}|$, because
  for every $\eps>0$  and every  a {\em c\`adl\`ag} function there exist a finite partition of $[0,T]$ into intervals on which the modulus of continuity is less than $\eps$ (see e.g. \cite{billingsley99convergence}). Since process $(W\topp q_t)$ is continuous in probability,
   $$\sup_{S<t\leq T}|W\topp q_{t-}-W\topp q_{t}|= \sup_{S\leq t\leq T}|W\topp q_{t-}-W\topp q_{t}|$$
   with probability one.  Thus
 \eqref{EQ:large-jumps} follows from \eqref{EQ:lim}.

\end{Rem}
\section{Connection to free probability}\label{sec:biane}
In this section, we explain how the tangent processes $\Z\topp{1/2},\Z\topp 1$ are connected to free probability. For this purpose, we first recall the notion of free convolution and free-convolution semigroup in free probability.
Free convolution of measures is a free-probability analog of the convolution of measures. While convolution describes the law of the sum of independent random variables, free convolution
  describes that law  of the sum of free noncommutative random variables.
  Both operations can also be introduced %
  analytically: convolution corresponds to  multiplication of the characteristic functions,
   and free convolution corresponds to addition of the so called the $R$-transforms.

To recall the analytic definition of free convolution,
denote by
$$G_\nu(z)=\int \frac{\nu(dx)}{z-x}$$ the
Cauchy-Stieltjes transform of a probability measure $\nu$ on the Borel sets of the real line.
It is known that $G_\mu$ is a well defined analytic function in the complex upper plane $z\in\CC^+=\{z=x+iy: y >0\}$ with  the right inverse
$K_\nu(z)=G^{-1}_\nu(z)$  which is well defined for  $z$ in a Stolz cone of the form $\{z=x+iy: |x|<\alpha y, |z|\leq \beta\}$. The $R$-transform %
of the probability measure $\nu$
 is then defined as
\begin{equation}\label{R def} R_\nu(z)=K_\nu(z)-1/z \end{equation} and the free convolution of %
two measures $\mu$ %
and $\nu$ is a (unique) probability measure, denoted %
by
$\mu\boxplus\nu$ with the $R$-transform
$R_{\mu}(z)+R_\nu(z)$ on the common domain.
These results,   %
 at increasing
levels of generality, have been established by  \citet{voiculescu86addition},  \citet{maassen92addition},  and   \citet{bercovici93free}.

 A free-convolution semigroup $\{\nu_t\}_{t\ge0}$  is the family of measures such that $\nu_{t+s}=\nu_t\boxplus\nu_s$, with $\nu_0 = \delta_0$ is a degenerate measure.
 For example, the family of Cauchy measures %
 \equh\label{eq:Cauchy}
\nu_t\topp 1(dx) = \frac{t}{\pi(x^2+t^2)}dx, t>0,
\eque
with Cauchy-Stieltjes transforms $G_{\nu_t\topp 1}(z)=1/(z+it)$ and  $R_{\nu_t\topp1}(z)=it$ on $\CC^+$, is a free-convolution semigroup, see \cite[Section 7]{bercovici93free} or \cite[Example 5.1]{biane98processes}.

 In the seminal paper \citep{biane98processes}, Biane associated to every free-convolution semigroup $\{\nu_t\}_{t\ge0}$ a classical Markov process $\{\Z_t\}_{t\ge0}$ such that the marginal distribution at time $t$ is $\nu_t$, and the transition probabilities
 $Q_{s,t}(x,dy)$ are determined as follows.
 Fix $s<t$ and $x\in\R$. Let $F$ be an analytic function on $\CC\setminus\RR$   such that
 \begin{equation}
   \label{Biane1}
  \int \frac{\nu_t(dx)}{z-x}=\int \frac{\nu_s(dx)}{F(z)-x} \mbox{ for $z\in\CC^+$ }.
 \end{equation}
 (Note that $F$ depends on $s<t$, but not $x$.)
 \citet{biane98processes} proved that such mapping exists and  is uniquely determined by the requirements that
 \begin{equation}
   \label{F-unique}
   \mbox{$F(\bar z)=\overline{F(z)}$, $F(\CC^+)\subset\CC^+$, $\Im F(z)\geq \Im z$ and $\lim_{y\to\infty}\frac{F(iy)}{iy}=1$}.
 \end{equation}
Furthermore, Biane showed that  $\CC^+\ni z\mapsto 1/(F(z)-x)\in\CC^-$ is a Cauchy--Stieltjes transform, so it defines a unique
probability measure  $Q_{s,t}(x,dy)$ such that
\begin{equation}
  \label{Biane2}
\int \frac{1}{z-y} Q_{s,t}(x,dy)=\frac{1}{F(z)-x}.
\end{equation}
The probability measures $\{Q_{s,t}(x,dy): s\leq t, x\in\RR\}$ satisfy Chapman--Kolmogorov equations, are Feller (i.e.~the map
$x\mapsto Q_{s,t}(x,dy)$ is weakly continuous) and $Q_{0,t}(0,dy)=\nu_t(dy)$; hence they are transition probabilities of a
Markov process, denoted by $\{\Z_t\}_{t\ge0}$.
We refer to the so-determined  Markov process $\{\Z_t\}_{t\ge0}$ as the {\it Biane process} associated to the free-convolution semigroup $\{\nu_t\}_{t\ge0}$.

Now recall the processes $\Z\topp1$ and $\Z\topp{1/2}$ described %
in
Section~\ref{sec:tangent}.
First, for the Cauchy process $\Z\topp 1$, it is well known that Cauchy distribution generates also the free $1$-stable semigroup
and by  \citep[Section 5.1]{biane98processes} the Cauchy process is indeed the Markov process associated to the free $1$-stable semigroup  \eqref{eq:Cauchy}.
So the Cauchy process is the 1-stable Biane process.
Second, the free 1/2-stable semigroup density appears in
\citet[page 1054]{Bercovici99stable}, see also \citep[Example 3.2]{perezabrew08free}.
The corresponding free-convolution semigroup of measures is then easily determined from rescaling, which gives
\begin{equation}\label{half-stable}
\nu_t\topp{1/2}(dx)=\frac{t\sqrt{4x-t^2}}{2\pi x^2}1_{(t^2/4,\infty)}dx, t>0.
\end{equation}
We show that $\Z\topp{1/2}$ defined by~\eqref{eq:biane} is the
Biane
 process
associated to $\{\nu\topp{1/2}_t\}_{t\ge0}$.

 \begin{Prop}\label{P:biane}The Biane process associated to~\eqref{half-stable} is $\Z\topp{1/2}$.
 \end{Prop}
\begin{proof}
To  determine transition probabilities of the Markov process $\Z_t$, we start from the Cauchy--Stieltjes transform
  \begin{equation}
    \label{CSK}
G_t(z):=\int \frac{\nu\topp{1/2}_t(dx)}{z-x}=
\frac{ t \sqrt{t^2-4 z}-t^2+2 z}{2z^2}
  \end{equation}
  of the free $1/2$-stable law \eqref{half-stable}.
  The Cauchy--Stieltjes transform of the closely related measure $\mu_t(A)=\nu\topp{1/2}_t(-A), A\in\calB(\R)$ appears explicitly in \citep[page 590]{bryc11one}.
 We will present a
 straightforward
 calculation of~\eqref{CSK} using basic complex analysis at the end of this section.

Next, we use the standard branch of the square root, and \eqref{CSK} simplifies to
\begin{equation}
  \label{CSK+}
  G_t(z)=-\frac{4}{\left(\sqrt{t^2-4 z}+t\right)^2}.
\end{equation}
The latter is the most convenient form for the equation \eqref{Biane1} which says  that $G_t(z)=G_s(F(z))$. Using \eqref{CSK+}, we first solve
$(\sqrt{t^2-4z}+t)^2 = (\sqrt{s^2-4F(z)}+s)^2$ for real $z<t/2$ fixed, seeking the real negative solution $F(z)<s/2$. The equation becomes
$$
t-s+\sqrt{t^2-4z}=\sqrt{s^2-4F(z)}.
$$
Since $s<t$, both sides are positive, so we get
\begin{equation}
  \label{F-sol}
  F(z)=\frac14\bb{s^2-\left(t-s+\sqrt{t^2-4z}\right)^2}.
\end{equation}
Formula \eqref{F-sol} has a unique analytic extension to all complex $z$ from the slit plane $\CC\setminus[t^2/4,\infty)$; the extension amounts to choosing the standard branch of the square root.
One can check that with this choice of the root, $F(z)$ given by \eqref{F-sol}  satisfies the uniqueness conditions \eqref{F-unique}.
Therefore,  \eqref{Biane2} determines the transition probabilities of the Markov process $(\Z_t)$ and  specifies their Cauchy--Stieltjes transform as
$$
\frac{4}{ s^2-4x -(t-s+\sqrt{t^2-4z})^2}.
$$
 The calculations turn out to be easier if we work with the process $\{\Z\topp{1/2}_{2t}-t^2\}_{t\ge0}$ by recasting \eqref{Biane2} via changing the variables in the above Cauchy--Stieltjes transform, first by replacing $s,t$ by $2s,2t$ and then replacing $x$ by $s^2+x$ and $z$ by $z+t^2$.
 This results in a somewhat simpler identity
\begin{equation}
  \label{Biane3}
  \int_{0}^\infty \frac1{z-y} p\topp{1/2}_{s,t}(x,y) dy= \frac{1}{ -x -(t-s+\sqrt{-z})^2} =: H_{s,t,x}(z)
\end{equation}
that we need to prove, with
$p\topp{1/2}_{s,t}(x,y)$ as in~\eqref{eq:p1/2}.
One way to verify \eqref{Biane3} is to apply the Stieltjes inversion formula and show:
$$-\frac{1}{\pi}\lim_{\eps\downarrow 0}\Im H_{s,t,x}(y+i \eps)=
p_{s,t}\topp{1/2}(x,y).
$$
This can be done by straight-forward calculation and is thus omitted.
\end{proof}

\begin{proof}[Proof of~\eqref{CSK}]
By self-similarity, it suffices to work with $t=1$. By definition,
\begin{multline*}
\int \frac{\nu_1\topp{1/2}(dx)}{z-x} = \int_{1/4}^\infty\frac{\sqrt{4x-1}}{2\pi x^2}\frac 1{z-x}dx
= \frac2\pi\int_0^\infty \frac{\sqrt y}{(y+1)^2}\frac 1{z-\frac{y+1}4}dy\\
= \frac4\pi\int_0^{\frac\pi2}\frac{\sin^2\alpha}{z-(4\cos^2\alpha)\inv}d\alpha = \frac1\pi\int_0^{2\pi}\frac{\sin^2\theta}{2(1+\cos\theta)z-1}d\theta,
\end{multline*}
where we used change of variables $4x-1\mapsto y, y\mapsto \tan^2\alpha, 2\alpha\mapsto \theta$ consecutively.
Transforming the last expression into a complex integral, we arrive at
\[
\int \frac{\nu_1\topp{1/2}(dx)}{z-x} = \frac1{\pi}\oint_{|\zeta|=1}\frac{\spp{\frac{\zeta-\wb\zeta}{2i}}^2}{2(1+\frac{\zeta+\wb\zeta}2)\zeta-1} \frac{d\zeta}{i\zeta} = -\frac1{4\pi i}\oint_{|\zeta|=1}\frac1z\frac{(\zeta^2-1)^2}{\zeta^2[\zeta^2+\frac{2z-1}z\zeta+1]}d\zeta.
\]
The integrand above has poles at
 \[
\zeta_0 = 0, \quad \zeta_1 = \frac{1+\sqrt{1-4z}}{1-\sqrt{1-4z}}\qmand \zeta_2 = \frac{1-\sqrt{1-4z}}{1+\sqrt{1-4z}},
 \]
 and $\zeta_0$ are $\zeta_2$ are within the unit disc for $z\in\C^+$ (we take the standard branch of square root).
We then write the complex integral as
 \[
 \oint_{|\zeta|=1}\frac{(\zeta^2-1)^2}{z\zeta^2(\zeta-\zeta_1)(\zeta-\zeta_2)}d\zeta =:\oint_{|\zeta|=1}h_z(\zeta)d\zeta,
 \]
 and obtain
 \[
 {\rm Res}_{\zeta_0}h_z = \frac1\zeta\pp{\frac1{\zeta_1}+\frac1{\zeta_2}} = \frac{1-2z}{z^2}\qmand {\rm Res}_{\zeta_1}h_z = \frac{(\zeta_2^2-1)^2}{z\zeta_2^2(\zeta_2-\zeta_2)} = -\frac{\sqrt{1-4z}}{z^2}.
 \]
The desired result now follows from the residue theorem:
\begin{multline*}
\int \frac{\nu\topp{1/2}_1(dx)}{z-x}= -\frac1{4\pi i}\oint_{|\zeta|=1}\frac1z\frac{(\zeta^2-1)^2}{\zeta^2[\zeta^2+\frac{2z-1}z\zeta+1]}d\zeta\\
= -\frac1{4\pi i}2\pi i ({\rm Res}_{\zeta_0}h_z + {\rm Res}_{\zeta_2}h_z) = \frac{\sqrt{1-4z}-1+2z}{2z^2}.
\end{multline*}
\end{proof}

\subsection*{Acknowledgements} WB thanks Chris Burdzy for pointing out  the close relation between the trajectories of the free Brownian motion  and the Cauchy process. YW's research was partially supported by NSA grant  H98230-14-1-0318.

\bibliographystyle{apalike}
\bibliography{references}

\end{document}